\documentclass{amsart}
\usepackage{amssymb}
\def\bu{$\bullet$\quad}
\def\CA{\mathcal A}
\def\CB{\mathcal B}
\def\Cal#1{\mathcal{#1}}
\def\CC{\mathbb C}
\def\CO{\mathcal O}
\def\DD{\mathbb D}
\def\eps{\varepsilon}
\def\Imm{\operatorname{Im}}
\def\Ree{\operatorname{Re}}
\def\wdtl{\widetilde}
\makeatletter
\def\th@mytheorem{%
  \let\thm@indent\noindent
  \thm@headfont{\bfseries}
    \itshape
}
\def\th@myremark{%
  \let\thm@indent\noindent
  \thm@headfont{\bfseries}
}
\makeatother
\theoremstyle{mytheorem}
\newtheorem*{Theorem*}{Theorem}
\theoremstyle{myremark}
\newtheorem*{Remark*}{Remark}
\begin{document}
\title{A counterexample to a theorem of Bremermann on Shilov boundaries}

\author[M.~Jarnicki]{Marek Jarnicki}
\address{Jagiellonian University, Faculty of Mathematics and Computer Science, Institute of Mathematics,
{\L}ojasiewicza 6, 30-348 Krak\'ow, Poland}
\email{Marek.Jarnicki@im.uj.edu.pl}

\author[P.~Pflug]{Peter Pflug}
\address{Carl von Ossietzky Universit\"at Oldenburg, Institut f\"ur Mathematik,
Postfach 2503, D-26111 Oldenburg, Germany}
\email{Peter.Pflug@uni-oldenburg.de}
\thanks{The research was partially supported by
grant no.~UMO-2011/03/B/ST1/04758 of the Polish National Science Center (NCN)}

\begin{abstract}
We give a counterexample to the following theorem of Bremermann on Shilov boundaries (\cite{Bre1959}): if $D$ is a bounded domain in $\CC^n$ having a univalent envelope of holomorphy,
say $\wdtl D$, then the Shilov boundary of $D$ with respect to the algebra $\CA(D)$, call it $\partial_SD$, coincides with the corresponding one for $\wdtl D$, called
$\partial_S\wdtl D$.
\end{abstract}



\maketitle

Let us first repeat some basic notions: let $D$ be a bounded domain in $\CC^n$. Put $\CA(D):=\Cal C(\overline D)\cap\CO(D)$. Then $\CA(D)$ together with the supremum norm
$\|\cdot\|_{\overline D}$ is a Banach algebra of functions on $\overline D$. Moreover, we set $\CB(D)$ as the closure of $\CO(\overline D)$ in $\Cal C(\overline D)$ with the above norm.
Again this is a Banach algebra. Then both of these algebras have a Shilov boundary, called $\partial_SD$, respectively, $\partial_BD$. Obviously, we have
$\partial_BD\subset\partial_SD$.

Now we assume in addition that $D$ has a univalent envelope of holomorphy $\wdtl D$. Then, we have
$$
\partial_S\wdtl D\subset\partial_SD\subset\partial D \text{ and } \partial_B\wdtl D\subset\partial_BD\subset\partial D.
$$
In \cite{Bre1959} (see Theorem in section 6.6), Bremermann  claimed that also equality $\partial_S\wdtl D = \partial_SD$
is true. While he proved $\partial_S\wdtl D\subset\partial_SD$ in detailed way he was saying that the inverse inclusion is an obvious fact. But as we will show (more than fifty years later) this claim is not true, already if $D$ is a Hartogs domain over an annulus.
For positive results regarding Reinhardt domains see \cite{KosZwo2013}.

Following the construction of some Hartogs domain with non univalent envelope of holomorphy (see \cite{JarPfl2000}, pages 1-2) we get the following result.

\begin{Theorem*} There exists a bounded Hartogs domain $D\subset\CC^2$ with a univalent envelope of holomorphy $\wdtl D$ such that

\bu $\partial_SD\neq\partial_S\wdtl D$,

\bu $\partial_BD\neq\partial_B\wdtl D$,

\bu there exists a function $f\in\CO(\overline D)$ such that the holomorphic extension $\wdtl f$ of $f|_D$ to $\wdtl D$ has no continuous extension to $\overline {\wdtl D}$.
\end{Theorem*}

\begin{proof} Let $A:=\{z\in\CC:1/2<|z|<1\}$. Then we introduce
\begin{align*}
D:&=\{z\in A\times\CC: \Ree z_1<0,\;|z_2|<3\}\\
&\cup\{z\in A\times\CC: 0\leq\Ree z_1,\; \Imm z_1>0,\; |z_2|<1\}\\
&\cup\{z\in A\times\CC: 0\leq\Ree z_1,\; \Imm z_1<0,\; 2<|z_2|<3\}.
\end{align*}
Note that $D$ is a Hartogs domain over the annulus $A$ which cuts the base $A\times\{0\}$. Then Corollary 3.1.10(b) in \cite{JarPfl2000} implies that $D$ has a univalent envelope of holomorphy $\wdtl D$. Moreover, using the Cauchy integral formula shows that $\wdtl D$ contains the following domain
$$
\{z\in A\times\CC: \Ree z_1<0 \text{ or (if } \Ree z_1\geq 0, \text{ then } \Imm z_1<0),\; |z_2|<3\}.
$$
In particular, all the discs $\DD_{z_1}:=\{z_1\}\times3\DD$, $1/2\leq\Ree z_1=z_1\leq 1$, belong to $\overline{\wdtl D}$ (here $\DD$ means the open unit disc in $\CC$).

Therefore, if $f\in \CA(\wdtl D)$, then $f(z_1,\cdot)\in \CA(3\DD)$ for all the above $z_1$ (use Weierstrass' theorem here). Hence, by the maximum principle, these discs don't contain any point of $\partial_S\wdtl D$.

On the other side we discuss the following domain
\begin{align*}
D':&=\{z\in A'\times\CC:\Ree z_1<\eps|\Imm z_1|,\; |z_2|<3+\eps\}\\
&\cup\{z\in A'\times\CC:0\leq\Ree z_1,\; \Imm z_1>-\eps\Ree z_1,\; |z_2|<1+\eps\}\\
&\cup\{z\in A'\times\CC:0\leq\Ree z_1,\; \Imm z_1<\eps\Ree z_1,\; 2-\eps<|z_2|<3+\eps\},
\end{align*}
where $0<\eps\ll 1/4$ and $A':=\{z\in\CC:1/2-\eps<|z|<1+\eps\}$. Observe that $D\subset\subset D'$.

Now we define the following concrete holomorphic function $g$ on $D'$:
$$
g(z):= \begin{cases}\log_1 z_1, &\text{ if } z\in D',\; |z_2|>1.4\\
\log_2z_1, &\text{ if } z\in D',\; |z_2|<1.6
\end{cases},
$$
where $\log_1$, respectively $\log_2$, is the branch of the logarithm function on $\CC\setminus \{w\in\CC:\Ree w\geq 0,\; \Imm w=\Ree w\}$, respectively on
$\CC\setminus\{w\in\CC: \Ree w\geq 0,\;\Imm w=-\Ree w\}$, with $\log_1(-1/2)=\log 1/2+i\pi$. Observe that $g$ is well defined on $D'$ and $g\in\CO(D')$.

Define $f:=g|_{\overline D}$. Then $f\in \CB(D)$ and  if $z\in\overline D$ with $z_1=\Ree z_1>0$, then
$$
f(z)=\begin{cases} \log\Ree z_1, &\text{ if } |z_2|\leq 1\\
                   \log\Ree z_1+2\pi i, &\text{ if }|z_2|\geq 2
\end{cases}.
$$

Finally, we observe that the function $h$ defined as $h(z):=\exp(i f(z)+2\pi)$, $z\in\overline D$, belongs to $\CB(D)$ and
$$
|h(z)|=\begin{cases} 1, &\text{ if } z\in\overline D,\; \Ree z_1>0,\; |z_2|\geq 2 \\
e^{2\pi}, &\text{ if } z\in\overline D,\; \Ree z_1>0,\;|z_2|\leq 1\\
\end{cases}
$$
and $|h(z)|<e^{2\pi}$ on the remaining part of $\overline D$. Therefore, $\partial_BD$ contains points $z\in\overline D$ with $0<\Ree z_1=z_1$ and
$|z_2|\leq 1$.

Combining this concrete information with the general one from the former discussion on the Shilov boundaries for $\wdtl D$  we conclude that $\partial_S\wdtl D$ and $\partial_B\wdtl D$ are strictly contained in $\partial_B D$. In particular, this shows that the claimed equality in Bremermann's paper does not hold.

Moreover, the function $f$ is the one whose existence was claimed in the third claim in the theorem.
\end{proof}

\begin{Remark*}
(a) Recall that the equality $\partial_SD=\partial_S\wdtl D$ is true for a Reinhardt domain $D$.

(b) Let $D$ be a bounded balanced domain. Obviously, $\overline D$ has a neighborhood basis of balanced domains $G$. Moreover, $D$ and each $G$ have univalent envelopes of
holomorphy with $\wdtl D\subset\subset\wdtl G$. Hence the equality $\partial_BD=\partial_B\wdtl D$ holds in an obvious way.

(c) What remains is to discuss the equality of the Shilov boundaries for $\CA(D)$ and $\CA(\wdtl D)$ in the case when $D$ is a balanced domain.
\end{Remark*}

\bibliographystyle{amsplain}

\end{document}